\documentclass{amsart}

\usepackage[english]{babel}
\usepackage{csquotes}
\usepackage[percent]{overpic}

\usepackage[style=alphabetic,
            isbn=false,
            doi=false,
            url=false,
            backend=bibtex,
            maxnames=5
           ]{biblatex}
\bibliography{sources}
\usepackage{amsmath,amsfonts,amsthm,mathtools, amssymb}
\usepackage{xcolor}
\usepackage{tikz-cd}
\usepackage{hyperref}
\usepackage{aliascnt}
\usepackage{xparse}
\usepackage{tensor}
\usepackage{caption}
\usepackage{ mathrsfs }
\usepackage{tabto}
\usepackage{xspace}
\usepackage{imakeidx}
\usepackage[normalem]{ulem}
\usepackage{ wasysym }

\makeatletter \let\@wraptoccontribs\wraptoccontribs
\def\@setauthors{%
  \begingroup
  \def\thanks{\protect\thanks@warning}%
  \trivlist
  \centering\footnotesize \@topsep30\p@\relax
  \advance\@topsep by -\baselineskip
  \item\relax
  \author@andify\authors
  \def\\{\protect\linebreak}%
  \MakeUppercase{\authors}%
  \ifx\@empty\contribs
  \else
    \linebreak \@setcontribs
    \@closetoccontribs
  \fi
  \endtrivlist
  \endgroup
}
\makeatother

 \RequirePackage{filecontents}

\indexsetup{level=\subsection*,toclevel=\paragraph,headers={ \MakeUppercase{Euler sequence on strata}}{\MakeUppercase{Euler sequence on strata}}, noclearpage,firstpagestyle=headings}
\makeindex[name=graph,title=Graphs and levels,columns=1,columnseprule,options=-s Idx.ist]

\usepackage{tikz}
\usetikzlibrary{matrix,arrows,calc}
\usetikzlibrary{positioning}
\usetikzlibrary{decorations.pathreplacing,decorations.markings,decorations.pathmorphing}
\usetikzlibrary{positioning,arrows,patterns}
\usetikzlibrary{cd}
\usetikzlibrary{intersections}

\tikzset{
	every loop/.style={very thick},
	comp/.style={circle,fill,black,,inner sep=0pt,minimum size=5pt},
	order bottom left/.style={pos=.05,left,font=\tiny},
	order top left/.style={pos=.9,left,font=\tiny},
	order bottom right/.style={pos=.05,right,font=\tiny},
	order top right/.style={pos=.9,right,font=\tiny},
	order node dis/.style={text width=.75cm},
	circled number/.style={circle, draw, inner sep=0pt, minimum size=12pt},
	below left with distance/.style={below left,text height=10pt},
    below right with distance/.style={below right,text height=10pt}
	}

\makeatletter
\ifcsname phantomsection\endcsname
    \newcommand*{\@gobblenexttocentry}[9]{}
\else
    \newcommand*{\@gobblenexttocentry}[4]{}
\fi
\newcommand*{\addsubsection}{%
    \addtocontents{toc}{\protect\@gobblenexttocentry}%
    \subsection*}
\makeatother

\DeclareRobustCommand{\SkipTocEntry}[9]{}

\begin{document}

\def\subsectionautorefname{Section}
\def\subsubsectionautorefname{Section}
\def\sectionautorefname{Section}
\def\equationautorefname~#1\null{(#1)\null}

\newcommand{\mynewtheorem}[4]{
  \if\relax\detokenize{#3}\relax 
    \if\relax\detokenize{#4}\relax 
      \newtheorem{#1}{#2}
    \else
      \newtheorem{#1}{#2}[#4]
    \fi
  \else
    \newaliascnt{#1}{#3}
    \newtheorem{#1}[#1]{#2}
    \aliascntresetthe{#1}
  \fi
  \expandafter\def\csname #1autorefname\endcsname{#2}
}

\mynewtheorem{theorem}{Theorem}{}{section}
\mynewtheorem{lemma}{Lemma}{theorem}{}
\mynewtheorem{rem}{Remark}{lemma}{}
\mynewtheorem{prop}{Proposition}{lemma}{}
\mynewtheorem{cor}{Corollary}{lemma}{}
\mynewtheorem{definition}{Definition}{lemma}{}
\mynewtheorem{question}{Question}{lemma}{}
\mynewtheorem{assumption}{Assumption}{lemma}{}
\mynewtheorem{example}{Example}{lemma}{}


\def\defbb#1{\expandafter\def\csname b#1\endcsname{\mathbb{#1}}}
\def\defcal#1{\expandafter\def\csname c#1\endcsname{\mathcal{#1}}}
\def\deffrak#1{\expandafter\def\csname frak#1\endcsname{\mathfrak{#1}}}
\def\defop#1{\expandafter\def\csname#1\endcsname{\operatorname{#1}}}
\def\defbf#1{\expandafter\def\csname b#1\endcsname{\mathbf{#1}}}

\makeatletter
\def\defcals#1{\@defcals#1\@nil}
\def\@defcals#1{\ifx#1\@nil\else\defcal{#1}\expandafter\@defcals\fi}
\def\deffraks#1{\@deffraks#1\@nil}
\def\@deffraks#1{\ifx#1\@nil\else\deffrak{#1}\expandafter\@deffraks\fi}
\def\defbbs#1{\@defbbs#1\@nil}
\def\@defbbs#1{\ifx#1\@nil\else\defbb{#1}\expandafter\@defbbs\fi}
\def\defbfs#1{\@defbfs#1\@nil}
\def\@defbfs#1{\ifx#1\@nil\else\defbf{#1}\expandafter\@defbfs\fi}
\def\defops#1{\@defops#1,\@nil}
\def\@defops#1,#2\@nil{\if\relax#1\relax\else\defop{#1}\fi\if\relax#2\relax\else\expandafter\@defops#2\@nil\fi}
\makeatother

\defbbs{ZHQCNPALRVW}
\defcals{DOPQMNXYLTRAEHZKCFI}
\deffraks{apijklmnopqueR}
\defops{Exp, PGL,SL,SO, Sp,mod,Spec,Re,Gal,Tr,End,GL,Hom,PSL,H,div,Aut,rk,Mod,R,T,Tr,Mat,Vol,MV,Res,vol,Z,diag,Hyp,ord,Im,ev,U,dev,c,CH,fin,pr,Pic,lcm,ch,td,LG,id,Sym,Aut}
\defbfs{abkiuvzwp} 

\def\ep{\varepsilon}
\def\abs#1{\lvert#1\rvert}
\def\dd{\mathrm{d}}
\def\inj{\hookrightarrow}
\def\eq{=}
\newcommand{\hyp}{{\rm hyp}}
\newcommand{\odd}{{\rm odd}}
\newcommand{\hor}{{\rm hor}}

\def\i{\mathrm{i}}
\def\e{\mathrm{e}}
\def\st{\mathrm{st}}
\def\ct{\mathrm{ct}}

\def\uC{\underline{\bC}}
\def\ol{\overline}
  
\def\Vrel{\bV^{\mathrm{rel}}}
\def\Wrel{\bW^{\mathrm{rel}}}
\def\twolev{\mathrm{LG_1(B)}}

\def\be{\begin{equation}}   \def\ee{\end{equation}}     \def\bes{\begin{equation*}}    \def\ees{\end{equation*}}
\def\ba{\be\begin{aligned}} \def\ea{\end{aligned}\ee}   \def\bas{\bes\begin{aligned}}  \def\eas{\end{aligned}\ees}
\def\={\;=\;}  \def\+{\,+\,} \def\m{\,-\,}

\newcommand*{\proj}{\mathbb{P}}
\newcommand{\barmoduli}[1][g]{{\overline{\mathcal M}}_{#1}}
\newcommand{\moduli}[1][g]{{\mathcal M}_{#1}}
\newcommand{\omoduli}[1][g]{{\Omega\mathcal M}_{#1}}
\newcommand{\modulin}[1][g,n]{{\mathcal M}_{#1}}
\newcommand{\omodulin}[1][g,n]{{\Omega\mathcal M}_{#1}}
\newcommand{\zomoduli}[1][]{{\mathcal H}_{#1}}
\newcommand{\barzomoduli}[1][]{{\overline{\mathcal H}_{#1}}}
\newcommand{\pomoduli}[1][g]{{\proj\Omega\mathcal M}_{#1}}
\newcommand{\pomodulin}[1][g,n]{{\proj\Omega\mathcal M}_{#1}}
\newcommand{\pobarmoduli}[1][g]{{\proj\Omega\overline{\mathcal M}}_{#1}}
\newcommand{\pobarmodulin}[1][g,n]{{\proj\Omega\overline{\mathcal M}}_{#1}}
\newcommand{\potmoduli}[1][g]{\proj\Omega\tilde{\mathcal{M}}_{#1}}
\newcommand{\obarmoduli}[1][g]{{\Omega\overline{\mathcal M}}_{#1}}
\newcommand{\obarmodulio}[1][g]{{\Omega\overline{\mathcal M}}_{#1}^{0}}
\newcommand{\otmoduli}[1][g]{\Omega\tilde{\mathcal{M}}_{#1}}
\newcommand{\pom}[1][g]{\proj\Omega{\mathcal M}_{#1}}
\newcommand{\pobarm}[1][g]{\proj\Omega\overline{\mathcal M}_{#1}}
\newcommand{\pobarmn}[1][g,n]{\proj\Omega\overline{\mathcal M}_{#1}}
\newcommand{\princbound}{\partial\mathcal{H}}
\newcommand{\omoduliinc}[2][g,n]{{\Omega\mathcal M}_{#1}^{{\rm inc}}(#2)}
\newcommand{\obarmoduliinc}[2][g,n]{{\Omega\overline{\mathcal M}}_{#1}^{{\rm inc}}(#2)}
\newcommand{\pobarmoduliinc}[2][g,n]{{\proj\Omega\overline{\mathcal M}}_{#1}^{{\rm inc}}(#2)}
\newcommand{\otildemoduliinc}[2][g,n]{{\Omega\widetilde{\mathcal M}}_{#1}^{{\rm inc}}(#2)}
\newcommand{\potildemoduliinc}[2][g,n]{{\proj\Omega\widetilde{\mathcal M}}_{#1}^{{\rm inc}}(#2)}
\newcommand{\omoduliincp}[2][g,\lbrace n \rbrace]{{\Omega\mathcal M}_{#1}^{{\rm inc}}(#2)}
\newcommand{\obarmoduliincp}[2][g,\lbrace n \rbrace]{{\Omega\overline{\mathcal M}}_{#1}^{{\rm inc}}(#2)}
\newcommand{\obarmodulin}[1][g,n]{{\Omega\overline{\mathcal M}}_{#1}}
\newcommand{\LTH}[1][g,n]{{K \overline{\mathcal M}}_{#1}}
\newcommand{\PLS}[1][g,n]{{\bP\Xi \mathcal M}_{#1}}

\DeclareDocumentCommand{\LMS}{ O{\mu} O{g,n} O{}}{\Xi\overline{\mathcal{M}}^{#3}_{#2}(#1)}
\DeclareDocumentCommand{\Romod}{ O{\mu} O{g,n} O{}}{\Omega\mathcal{M}^{#3}_{#2}(#1)}

\newcommand*{\Tw}[1][\Lambda]{\mathrm{Tw}_{#1}}  
\newcommand*{\sTw}[1][\Lambda]{\mathrm{Tw}_{#1}^s}  

\newcommand{\bfa}{{\bf a}}
\newcommand{\bfb}{{\bf b}}
\newcommand{\bfd}{{\bf d}}
\newcommand{\bfe}{{\bf e}}
\newcommand{\bff}{{\bf f}}
\newcommand{\bfg}{{\bf g}}
\newcommand{\bfh}{{\bf h}}
\newcommand{\bfm}{{\bf m}}
\newcommand{\bfn}{{\bf n}}
\newcommand{\bfp}{{\bf p}}
\newcommand{\bfq}{{\bf q}}
\newcommand{\bfP}{{\bf P}}
\newcommand{\bfR}{{\bf R}}
\newcommand{\bfU}{{\bf U}}
\newcommand{\bfu}{{\bf u}}
\newcommand{\bfv}{{\bf v}}
\newcommand{\bfz}{{\bf z}}

\newcommand{\bfl}{{\boldsymbol{\ell}}}
\newcommand{\bfmu}{{\boldsymbol{\mu}}}
\newcommand{\bfeta}{{\boldsymbol{\eta}}}
\newcommand{\bfomega}{{\boldsymbol{\omega}}}

\newcommand{\wh}{\widehat}
\newcommand{\wt}{\widetilde}

\newcommand{\ps}{\mathrm{ps}}  

\newcommand{\tdpm}[1][{\Gamma}]{\mathfrak{W}_{\operatorname{pm}}(#1)}
\newcommand{\tdps}[1][{\Gamma}]{\mathfrak{W}_{\operatorname{ps}}(#1)}

\newlength{\halfbls}\setlength{\halfbls}{.5\baselineskip}
\newlength{\halbls}\setlength{\halfbls}{.5\baselineskip}

\newcommand*{\Hrel}{\cH_{\text{rel}}^1}
\newcommand*{\Hrelbar}{\overline{\cH}^1_{\text{rel}}}

\newcommand*\interior[1]{\mathring{#1}}

\newcommand{\prodt}[1][\lceil j \rceil]{t_{#1}}
\newcommand{\prodtL}[1][\lceil L \rceil]{t_{#1}}

\defbbs{ZHQCNPALRVW}
\defcals{ABOPQMNXYLTRAEHZKCFI}
\deffraks{apijklgmnopqueRC}
\defops{IVC, PGL,SL,mod,Spec,Re,Gal,Tr,End,GL,Hom,PSL,H,div,Aut,rk,Mod,R,T,Tr,Mat,Vol,MV,Res,Hur, vol,Z,diag,Hyp,hyp,hl,ord,Im,ev,U,dev,c,CH,fin,pr,Pic,lcm,ch,td,LG,id,Sym,Aut,Log,tw,irr,discrep,BN,age,hor,lev,Per}
\defbfs{uvzwp} 

\newcommand{\Teichmuller}{Teich\-m\"uller\xspace}


\title[Teichm\"uller curves in meromorphic strata]
      {Teichm\"uller curves in hyperelliptic components of meromorphic strata}

\author{Martin M\"oller}
\address{M.\ M\"oller: Institut f\"ur Mathematik, Goethe-Universit\"at Frankfurt,
Robert-Mayer-Str. 6-8,
60325 Frankfurt am Main, Germany}
\email{moeller@math.uni-frankfurt.de}
\thanks{Research of M.M is supported by the DFG-project MO 1884/2-1
  and the Collaborative Research Centre
TRR 326 ``Geometry and Arithmetic of Uniformized Structures''.
}

\author{Scott Mullane}
\address{\noindent S. Mullane: Humboldt-Universit\"at zu Berlin, Institut f\"ur Mathematik,  Unter den Linden 6, 10099 Berlin, Germany} \email{{scott.mullane@hu-berlin.de}}
\thanks{Research of S.M is supported by the Alexander von Humboldt Foundation, and the ERC Advanced Grant ``SYZYGY''}

\contrib[with an appendix by]{Benjamin Bakker}
\address{B. Bakker: Dept. of Mathematics, Statistics, and Computer Science, University of Illinois at Chicago, Chicago, USA.}
\email{bakker.uic@gmail.com}

\contrib[]{Scott Mullane}

\begin{abstract}
We provide a complete classification of \Teichmuller curves occurring in hyperelliptic components of the meromorphic strata of differentials. 
Using a non-existence criterion based on how \Teichmuller curves intersect the boundary of the moduli space we derive a contradiction to the algebraicity of any candidate outside of Hurwitz covers of strata with projective dimension one, and Hurwitz covers of zero residue loci in strata with projective dimension two. 
\end{abstract}
\maketitle
\tableofcontents

\section{Introduction} \label{sec:intro}

\Teichmuller curves are usually defined as immersed curves $C \to \moduli[g]$
in the moduli spaces of curves that are totally geodesic for the
\Teichmuller metric. They are generated by an abelian or quadratic
differential on any of the Riemann surfaces~$X$ parameterized by the curve~$C$.
Passing to a double cover of~$X$ we may (and we will) restrict to the case
of curves generated by an abelian differential~$\omega$. A \Teichmuller curve
thus defines the \emph{type} $\mu =(m_1,\ldots,m_n)$ of the abelian
differential~$\omega$,
a tuple of integers with sum equal to $2g-2$, the order of zeros of~$\omega$.
The classical case, originating
from a discovery of Veech \cite{veech89} thus deals with a differential of
\emph{holomorphic type} where all $m_i \geq 0$ and has beautiful connections
to billiards. There are several infinite series of \Teichmuller curves
(\cite{ward,mcmullenbild,calta,mcmullenprym,bouwmoel,MMWGothic}),
a complete classification in low genus (\cite{mcmullentor})
and finiteness results (\cite{MatWri, BHM16}) thanks to input to this geometric
problem from Hodge theory and number theory. See also \cite{McMTsurv} for
the most recent survey and a lot of open questions on \Teichmuller curves.
\par
In this paper we shed some light into the case of \emph{meromorphic
differentials}, i.e.\ the case where at least one of the $m_i$ is negative.
There are several equivalent definitions of \Teichmuller curves that
we briefly recall in Section~\ref{sec:charteich}. Relevant for us is the
following
characterization, that we take as the definition in the meromorphic case:
A \emph{\Teichmuller curve} is an immersed algebraic curve  $C \to
\moduli[g]$ which is the the image under the forgetful map of a
$2$-dimensional variety~$M \to \omoduli[g](\mu)$ in the moduli space
of flat surfaces of type~$\mu$, which is locally cut out by $\bR$-linear
equations in the period coordinates. The Appendix
provides an example why algebraicity is a non-trivial additional condition
in the meromorphic case.
\par
Removing the dimension hypothesis in this definition we arrive at the
notion of \emph{linear manifold} (also known as \emph{affine invariant
submanifold}). In the holomorphic case these are the closures of
$\GL_2(\bR)$-orbits by the fundamental results of Eskin-Mirzakhani
and Mohammadi (\cite{EsMi, EsMiMo}). The classification of linear manifolds
has recently attracted a lot of attention,
both by exhibiting exceptional examples (\cite{MMWGothic, emmw}) and by
deriving constraints to the existence (e.g.\ \cite{MWfull,AWmarked}).
These constraints are often derived by
degeneration arguments, either to the boundary of Mirzakhani-Wright
(\cite{MWboundary}) or, retaining even more information, to the multi-scale
compactification (\cite{BCGGM3}). Recent work of Benirschke-Dozier-Grushevsky
\cite{BDG} states that the boundary intersection of linear manifolds is (roughly)
a product of linear manifolds, now also in meromorphic strata.
Exploring the possibilities for such boundary intersections is a main
motivation for our classification attempts.
\par
Throughout this paper we restrict our attention to the hyperelliptic strata.
(We recall Boissy's classification of connected components of the meromorphic
strata in Section~\ref{sec:components}.) Just as in McMullen's genus two
classification \cite{mcmullenspin}, the first classification result in the
holomorphic case, we consider hyperelliptic strata to simply reduce
the combinatorial complexity.
\par
Certain obvious sources of \Teichmuller curves exist in
meromorphic strata. They arise as Hurwitz spaces of covers of strata
whose projectivized dimension is one, or whose projectivized dimension
is one after imposing conditions on the residues. We compile in
Proposition~\ref{prop:obviousTeich} the rather short list of
those \emph{obvious \Teichmuller curves} that lie in hyperelliptic strata.
This list is analagous to square-tiled surfaces in the
holomorphic case: They form an infinite series,  the degree of the
cover being one obvious invariant, and the precise classification of
irreducible components is probably a tedious task. Our main result is:
\par
\begin{theorem} \label{thm:main}
The only \Teichmuller curves in a hyperelliptic stratum of
meromorphic differentials are obvious \Teichmuller curves.
\end{theorem}
\par
To prove this theorem we need to show that the linear manifold~$M$
containing the $\GL_2(\bR)$-orbit of a given flat surface is not closed
and algebraic unless we are in one of the obvious cases. Providing
a neighborhood of a point in the stratum or in an algebraic compactification
that intersects~$M$ in infinitely many irreducible components would suffice
to rule out a flat surface, just as it is done in the appendix. In practice
we found this difficult to achieve. Instead we design a criterion
(Proposition~\ref{prop:nonTeichCrit}) based on the structure of
boundary intersections of linear manifolds in \cite{BDG}: It suffices
to exhibit a surface in~$M$ with a cylinder (in say, the vertical direction)
and a non-vertical saddle connection outside all vertical cylinders. To apply this
criterion, we use several paths in~$M$, nicknamed 'complex conjugation',
'coordinate dancing' and 'pulling through cylinders', see
Section~\ref{sec:higherpole}. Interestingly, the paths most useful to
derive constraints on flat surfaces $(X,\omega)$ generating \Teichmuller
curves stay in~$M$ but leave the closure of the $\GL_2(\bR)$-orbit of $(X,\omega)$!
\par
Finally, we note that non-obvious \Teichmuller curves do exist in other strata. Examples
are given by closure of the Gothic locus \cite{MMWGothic} in the multi-scale
compactification \cite{BCGGM3}, more precisely the lower level components
of some boundary strata, as explained in detail in \cite{Schwab}.
Moving forward, it would be interesting to obtain a more conceptual
understanding of \Teichmuller\ curves in meromorphic strata, similar
to what is known in the holomorphic case.


\section{Obvious Teichm\"uller curves} \label{sec:background}

This section collects the background material on \Teichmuller curves,
on components of strata of meromorphic differentials and provides
a classification of 'obvious' \Teichmuller curves in hyperelliptic
components of strata. We assume that the reader is familiar
with basic notions about strata of differentials and flat surfaces, such
as period coordinates, the $\GL_2(\bR)$-action and Veech groups. Reference
for this includes the surveys \cite{zorich06,FiSurvey}.

\subsection{Characterizations of \Teichmuller curves}
\label{sec:charteich}

A \emph{Teichm\"uller curve  in a holomorphic stratum} $\omoduli[g](\mu)$
of the moduli space of flat surfaces generated by a
flat surface $(X,\omega)$ admits several equivalent characterizations.
\par
\begin{prop}
A map $C \to \moduli[g]$ from a complex curve to the moduli space
of curves is a \emph{Teichm\"uller curve} generated by a holomorphic abelian
differential if one of the following equivalent conditions hold
\begin{itemize}
\item[i)] The map is an immersion of a \emph{totally geodesic curve}
whose Teichm\"uller maps are generated by a quadratic differential
$q=\omega^2$ which is a square of an abelian differential.
\item[ii)] The curve~$C$ is the quotient by $\SO_2(\bR)$ of the
orbit $\SL_2(\bR) \cdot (X,\omega)$ of a flat surface which is
\emph{closed} in $\omoduli[g](\mu)$.
\item[iii)] The curve~$C$ is the image of the $\SL_2(\bR)$-orbit 
of a flat surface whose Veech group is a lattice in $\SL_2(\bR)$.
\item[iv)] The curve~$C$ is the image of a $2$-dimensional
subvariety~$M$ defined by $\bR$-linear equations in the period coordinates
of $\omoduli[g](\mu)$ under the forgetful map $\omoduli[g](\mu)
\to \moduli[g]$.
\item[v)] The variation of Hodge structures over~$C$ has a rank two
local subsystem which is maximal Higgs in the sense of \cite{moeller06}.
\end{itemize}
\end{prop}
\par
Note that this proposition does not suppose~$C$ to be algebraic.
Algebraicity for~$C$ is a general property of quotients of the
upper half plane by cofinite Fuchsian groups. Algebraicity of the
embedding can be seen as a consequence of Chow's theorem or, including
the case of higher dimensional linear manifolds, of Filip's theorem
\cite{Filip}.
\par
\begin{proof}
The equivalence of~(i) and~(ii) is a consequence of \Teichmuller's
theorem and the fact that the \Teichmuller metric is the Kobayashi
metric, see \cite{mcmullenbild}.
The equivalence of~(ii) and~(iii) is shown by Smillie-Weiss \cite{swminimal}.
The equivalence of~(ii) and~(iv) is nearly a tautology, passing
from the $\SL_2(\bR)$-orbit to the $\GL_2^+(\bR)$-orbit.
The equivalence of~(v) and~(iii) is the main content of~\cite{moeller06}.
\end{proof}
\par
In meromorphic strata we recall that (ii) and (iii) do not give
interesting classes of objects in meromorphic strata.
\par
First, as noticed by Valdez in \cite{ValdezVG}, the Veech group of a
meromorphic flat surface is (up to conjugation and $\pm {\rm Id}$) either
a finite subgroup of the rotation group, a cyclic parabolic group, or
a $2$-dimensional Lie group, the stabilizer of $(1,0)^T$ in $\GL_2^+(\bR)$.
Hence an analog of condition~(iii) never occurs.
\par
Recall that the \emph{core $\cC(X)$} of a flat surface~$(X,\omega)$ is
the convex hull of the saddle connections of $(X,\omega)$. It is a polygon
in~$X$ bounded by saddle connections and possibly with empty interior.
\par
Second, as noticed by Tahar in \cite{TaharVG}, the $\GL_2^+(\bR)$-orbit of
a meromorphic flat surface is closed if and only if all of its saddle
connections are parallel. Such surfaces are easy to construct, abundant
but nowhere dense in a stratum, and the $\GL_2^+(\bR)$-orbits are just
$\bC^*$-orbits, linear of dimension~$1$ in period coordinates.
\par
Tahar also remarks that surfaces with closed $\SL_2(\bR)$-orbit are abundant,
namely where the boundary of the core contains two linearly independent saddle
connections (and the Veech group is trivial), or where the core consists of
a collection of cylinders with commensurable moduli (and the Veech group is
cyclic parabolic). Again, these are abundant and easy to construct.
We thus do not consider condition~(ii). 
\par
We thus define a \emph{Teichm\"uller curve in a meromorphic stratum}
in analogy to condition~(iv), as in the introduction, including
the algebraicity hypothesis. (See the appendix for an examples where
this algebraicity condition is violated.)
We also call (slightly abusing dimension
notation) a Teichm\"uller curve the two-dimensional linear manifold
$M \to \omoduli[g](\mu)$ in a stratum of meromorphic differentials.
Let $(X,\omega)$ be a flat surface in~$M$. Contrary to the holomorphic case,
this $\GL_2^+(\bR)$-orbit is never equal to~$M$, as it sweeps out only
one of the chambers of the Teichm\"uller curve bounded by loci of parallel
saddle connections.
\par
We call a meromorphic flat surface $(X,\omega)$ a \emph{(meromorphic) Veech
surface}, if its $\GL_2^+(\bR)$-orbit is contained in a \Teichmuller
curve~$M$ and equal to~$M$ on an open subset of~$M$. We say that $(X,\omega)$
is \emph{generated by} $(X,\omega)$ in this case.
\par
It would be interesting to have a characterization of \Teichmuller curves
in metric terms or Hodge theory, as in~(i) or~(v).

\subsection{Components of strata} \label{sec:components}

Boissy classified in \cite{boissymero} the connected component of meromorphic
strata. A \emph{hyperelliptic component} in a stratum $\omoduli[g](\mu)$ is a
component that consists exclusively of hyperelliptic curves.
Recall from \cite{boissymero} that a signature $\mu$ is called of
\emph{hyperelliptic type}, if the polar part is of the form $\{-p,-p\}$
or $\{-2p\}$ for some $p \in \bN$ and if the zero part is of the
form $\{m,m\}$ or $\{2m\}$ for some $m \in \bN$. Then Boissy shows that
hyperelliptic components exist in meromorphic strata $\omoduli[g,n](\mu)$
precisely if the signature is of hyperelliptic type.
(The full classification of components distinguishes moreover the
spin parity and in genus one the divisibility of the rotation numbers,
see \cite{boissymero} for details.)
%

\subsection{The obvious examples} \label{sec:obviousexamples}

In this section we classify the obvious examples of  \Teichmuller curves
in hyperelliptic strata: We call a \Teichmuller curve \emph{obvious}, if
it is the intersection of a Hurwitz space with a locus prescribed by
residue conditions as follows.
\par
Recall that meromorphic strata admit a \emph{residue map} ${\rm res}:
\omoduli[g](\mu)
\to \bC^p$ defined by integrating cycles around the poles. Here $p$ is the
number of negative entries in~$\mu$. By the residue theorem the image is
contained in the hypersurface
$\bC^p_{\rm res}$ where the coordinates sum to zero. Since the residue map is
an algebraic morphism, the preimage of ($\bR$-)linear subvarieties of
$\bC^p_{\rm res}$ are algebraic ($\bR$-)linear subvarieties of $\omoduli[g](\mu)$,
including Teichm\"uller curves, if the number of poles permits cutting down the
dimension sufficiently.
\par
In general meromorphic strata there is a zoo of obvious \Teichmuller curves
taking a genus zero signature~$\ol{\mu}$ with~$n$ entries, including $k$ higher
order poles and imposing $n-4$ linear conditions on the residues of these poles.
(This requires $k \geq n-3$ of course.) In hyperelliptic components, the
possibilities are rather limited:
\par
\begin{prop} \label{prop:obviousTeich}
The obvious \Teichmuller curves in hyperelliptic components of
meromorphic strata are Hurwitz spaces $\Hur(d,\ol{\mu})$ parameterizing
degree~$d$ covers 
\begin{itemize} 
\item[(i a)] of flat surfaces in the stratum
$\omoduli[0](\ol{\mu})$
with $\ol{\mu} = (m-1,m-1,-m,-m)$, fully ramified over both zeros
and both poles and unramified elsewhere, or
\item[(i b)] of flat surfaces in the stratum
$\omoduli[0](\ol{\mu})$ with $\ol{\mu} = (0,0,-1,-1)$,
fully ramified over both poles and the zeros having a unique
ramification point of the same order in their fibers, or
\item[(ii a)] of flat surfaces in the stratum $\omoduli[0](\ol{\mu})$
with $\ol{\mu} = (m,-m)$ and with $d$ odd, fully ramified over the zero
and the pole and unramified elsewhere.
\item[(ii b)] of flat surfaces in the stratum $\omoduli[0](\ol{\mu})$
with $\ol{\mu} = (m,-m)$  with $d$ even, splitting into two subcases
depending on whether over the zero and the pole there is full
ramification or two points of ramification order $d/2$ and unramified elsewhere.
\item[(iii a)] of flat surfaces in the residue-zero locus of
the stratum $\omoduli[1](\ol{\mu})$ with $\ol{\mu} = (m,-m/2,-m/2)$
and with $d$ odd, fully
ramified over the zero and the poles and unramified elsewhere, or
\item[(iii b)] of flat surfaces in the residue-zero locus of
the stratum $\omoduli[1](\ol{\mu})$ with $\ol{\mu} = (m,-m/2,-m/2)$ fully
ramified the poles and ramified to order $d/2$ over the zero and
unramified elsewhere.
\end{itemize}
\end{prop}
\par
\begin{proof}
The projectively one-dimensional strata without residue conditions
are $\ol{\mu} = (m_1,m_2,m_3,m_4)$ with $\sum m_i = -2$ in $g=0$
and $\ol{\mu} = (m,-m)$ in $g=1$. Each preimage of a zero (i.e., $m_i>0$) gives a
zero and each preimage of a pole (i.e., $m_i<0$) 
gives a pole, each of which there are at most two. Consequently
the possibilities for the zeros among $\ol{\mu}$ are one zero fully
ramified, one zero with two preimages ramified to order $d/2$, or
two zeros fully ramified. The analogous statement holds for the poles.
Moreover, the final possibility is  some $m_i=0$, but in this case there needs
to be ramification over it (otherwise the Hurwitz space is a point).
However, the ramification profile might now be one or two ramified
preimages with an arbitrary number of additional unramified sheets.
For $g=0$, excluding simple poles, this leaves only
the possibility $\ol{\mu} = (m-1,m-1,-m,-m)$ as in case (ia), since
there are four special points to be taken care of. If there is a
simple pole, it will be preserved under coverings and so there cannot
be any higher order pole, leaving only the possibility in case (ib).
For $g=1$ we derived all the restrictions listed in cases (iia) and (iib).
The cases of different ramification orders of the zero and the pole
are excluded by Riemann-Hurwitz.
\par
Next consider strata with residue conditions. Since they are associated
with poles that don't disappear under coverings, there are at most two
poles involved and thus at most one residue condition. For $g=0$ the
tuple now has five entries, leading to at least three zeros or three
poles, impossible for hyperelliptic strata. For $g=1$ the only case
is $\ol{\mu} = (m_1,m_2,m_3)$ with $m_1 > 0 > m_2$ and $m_3 < 0$.
The presence of two poles implies full ramification over them and
thus $m_2 = m_3$ by hyperellipticity. This leaves the cases in (iiia)
and (iiib) only. 
\end{proof}
\par
The geometry of the strata $\omoduli(m,-m)$ has been studied in
detail in \cite{tahar} in terms of the wall-and-chamber decomposition
determined by the geometry of the core (the convex hull of the saddle
connections) of the meromorphic differential.
\par



\section{A non-existence criterion} \label{sec:nonexistence}

In this section we provide a criterion to rule out \Teichmuller curves
based on a degeneration statement in \cite{BDG}, which in turn is based
on the existence of a good compactification from \cite{BCGGM3}. To apply
this criterion, we start with a summary of how to present meromorphic flat
surfaces. The proof of the classification Theorem~\ref{thm:main} is
completed at the end of this section.

\subsection{Boissy's infinite zippered rectangle construction and a variant
with cylinders.} \label{sec:Boissy}

We briefly recall the presentation of meromorphic flat surfaces from
\cite[Section~3.3]{boissymero}. Boissy starts with a meromorphic flat
surface~$(X,\omega)$ oriented such that the vertical
direction does not admit any saddle connection. The result of his construction
is a decomposition of the surface into half-planes bounded by broken lines,
in fact saddle connections and two infinite horizontal separatricies, and
infinite cylinders with non-vertical core curves bounded by broken lines composed
of saddle connections. There is one such cylinder for each simple pole. This
datum can be encoded by the gluing combinatorics of the saddle connections and
separatricies,
as well as the periods of the saddle connections. These half planes and
infinite cylinders are called \emph{basic domains}. We refer to the saddle connections
on the boundary of these basic domains as \emph{boundary saddle connections} and
denote them by $v_i$ or $v_i^\pm$ if we need to specify the two segments after the
surface has been cut open. We write $\Per(v_i) := \int_{v_i} \omega$ for the period
of any saddle connection. Since these periods give local coordinates of the
stratum we also refers to boundary saddle connections briefly as \emph{coordinates}.
Conversely, given a collection of basic domains with side pairings and periods of the
boundary saddle connections allows to construct the surface uniquely.
\par
Boissy's algorithm decompose the surface starts with Strebel's classification
of the vertical trajectories. It provides a decomposition of the surface into
half-planes and half-infinite \emph{vertical} strips. Each left half plane has a single
singularity at its right boundary. It is cut open along the horizontal separatrix
starting there. The two pieces are the start of an upper and lower half plane
assembled by 
attaching vertical strips until the process terminates with a right half plane cut
along the horizontal separatrix starting at the singularity at its left boundary.
The remaining vertical strips are assembled similarly to infinite cylinders with
non-vertical core curves.
\par
Our variant allows meromorphic flat surfaces~$(X,\omega)$ with saddle connections
in vertical direction, but only if these bound cylinders with core curves in the
vertical direction. After removing these cylinders the rest of the surface can
be decomposed by following Boissy's algorithm above verbatim. It merely requires
to consistently choose how to treat vertical saddle connections when assembling the
basic domains: We orient the saddle connections on the boundary of the basic
domain working from left to right as above, and then require that the period
of vertical saddle connections has \emph{positive} imaginary part. This yields:
\par
\begin{prop} \label{prop:variantBoissy}
A meromorphic flat surface can be decomposed uniquely into a finite number
of finite area cylinders with vertical core curves and a finite collections
of Boissy's basic domains.
\end{prop}
\par
We refer to this as the \emph{generalized Boissy presentation}. Note that
the periods of boundary saddle connections will no longer give local coordinates:
some distinct  boundary saddle connections (at the ends of some cylinders)
may be exchanged by the hyperelliptic involution~$h$ and periods of $h$-orbits 
give coordinates. Nevertheless we abbreviate 'boundary saddle connection'
as 'coordinate'.

\subsection{Equations of linear manifolds at the boundary}

Suppose that $M \subset \omoduli(\mu)$ is an algebraic linear manifold and suppose
that it intersects the boundary stratum given by a level graph~$\Gamma$.
Our main criterion to rule out the existence \Teichmuller curves is
\cite[Theorem~1.5]{BDG}, that characterizes the generic points of
boundary components of the closure $\ol{M} \subset \LMS$. For
a \Teichmuller curve the boundary consists just of points up to the $\bC^*$-rescaling
of the differential. We may restate their criterion as:
\par
\begin{prop} \label{prop:BDGcrit}
  The boundary points of \Teichmuller curves have level graphs with
only horizontal nodes (i.e.\ just one level) or with two levels and
no horizontal nodes.
\end{prop}
\par
Our strategy to rule out that a flat surface~$(X,\omega)$ generates a
\Teichmuller curve is thus to exhibit a cylinder, such that shrinking
the core curve of the cylinder leads to a surface with a horizontal
node on lower level. Geometrically this is verified as follows.
\par
\begin{prop} \label{prop:nonTeichCrit}
Let $(X,\omega)$ be a  meromorphic flat surface that admits a cylinder with core
curve in the vertical directions. If the basic domains in the generalized
Boissy presentation are bounded by at least one  saddle connection~$\gamma$
that is not vertical, then $(X,\omega)$ does not generate a \Teichmuller curve.
\par
Equivalently, if  $(X,\omega)$ has a saddle connection in a non-vertical
direction outside the closures of the cylinders with vertical core curves, then
$(X,\omega)$ does not generate a \Teichmuller curve.
\end{prop}
\par
Said differently, candidates for meromorphic Veech surfaces are only
\emph{cylinder-free flat surfaces} and \emph{vertically presentable surfaces},
i.e.\ surfaces that have a cylinder, that we may assume to be vertical, and
then all boundaries of the basic domains are vertical, too.
\par
\begin{proof}
Consider the path in the $\GL_2(\bR)$-orbit of $(X,\omega)$ by shrinking
the vertical direction while maintaining the period of~$\gamma$ constant.
Using Proposition~\ref{prop:BDGcrit} we have to rule out that the resulting
degeneration has a level graph with one level and horizontal edges only.
If the initial flat surface is $(X,\omega) \in \omoduli[g,n](\mu)$, then the
normalization of such a limiting stable curve belongs to a stratum of
the form $\omoduli[g',n+2(g-g')](\mu,-1^{2(g-g')})$. (In particular the limiting
flat surface can be drawn entirely with Boissy's algorithm, there is no
component on which the differential tends to zero.) This implies that the ratio of
the residue of any of the simple pole and the length of any interior saddle
connection is bounded above and away from zero on any path approaching this
limit. This property is violated for the path we described initially.
\end{proof}

\subsection{Hyperelliptic components of meromorphic strata with simple poles}
\label{sec:simplepole}

We start our proof of Theorem~\ref{thm:main} with the easier subcase where
simple poles exist.
\par
\begin{prop}\label{prop:HYPtwosimplepoles}
  For $g\geq 1$, in the hyperelliptic strata with two simple poles $\omoduli[g,2](2g,-1,-1)$ there are no \Teichmuller curves, while in the hyperelliptic strata $\omoduli[g,2](g,g,-1,-1)$ the \Teichmuller curves are the irreducible components of the linear manifold
  $\Hur(d,(g,g,-1,-1))$ given in Proposition~\ref{prop:obviousTeich} case~(i b)
  for $d\geq g+1$.
\end{prop}
\par
Before proving this proposition, we provide a simple, but useful lemma concerning the characterization of linear manifolds in a Boissy presentation. This relates two surfaces that in general will not lie in the same $\GL_2^+(\bR)$-orbit but must lie in the same linear manifold.
\par
\begin{lemma}\label{lem:conjugation}
If a surface in a Boissy presentation with boundary saddle connections~$v_i$ and
given by their periods $(\Per(v_1),\dots,\Per(v_k))$  lies in a linear manifold~$M$,
then also the surface obtained in this presentation by complex conjugation of the
periods $(\overline{\Per(v_1)},\dots, \overline{\Per(v_k)})$ lies in~$M$.
\end{lemma}
\begin{proof}
The real path~$\phi(t) = (X_t,\omega_t)$ for $t\in[0,1]$  in the strata between the
two surfaces given by Boissy presentations with the same combinatorics and
\be \label{eq:pullthrough}
\Per_t(v_i) := \int_{v_i} \omega_t = v_i - 2t\Im(\Per({v}_i))
\ee
will satisfy all $\bR$-linear equations satisfied by $\phi(0)$.
\end{proof}
\par
\begin{proof}[Proof of Proposition~\ref{prop:HYPtwosimplepoles}]
Let $(X,\omega)$ be a meromorphic surface generating a \Teichmuller curve $M$.  As $(X,\omega)$ will be Boissy decomposable in some direction we can act by $\GL_2^+(\bR)$ to obtain a surface in $M$ that is decomposable in the horizontal direction, where the residue core curves are also horizontal and have length one. 
\par
That is, we can assume that $(X,\omega)$ is given by a Boissy presentation in
the horizontal direction where one of the infinite cylindrical domains has
coordinates $(v_1,\dots,v_k)$ labeled  from left to right and with
$\sum_{i=1}^k \Per(v_i)=1$.  Furthermore, the hyperelliptic involution then
necessitates that on the other domain the saddle connections are labeled
$(v_k,\dots,v_1)$ from left to right, as in Figure~\ref{cap:simplepole}.
\definecolor{color1}{rgb}{0.83, 0.83, 0.83} 
\begin{figure}[htb]
	\centering
\begin{tikzpicture} [scale=.8]
\def\y{0.9} 
\begin{scope}
  \path[draw,fill=color1]  (0,3) coordinate (P0)
  -- ++ (0,-3) coordinate (P1)  node[pos =0.5,rotate=90] {\tiny$\mathrm{I}$}
-- ++ (1,1) coordinate (P2) node[below, pos =0.6] {$v_1$}
-- ++ (1,0.5) coordinate (P3) node[below,pos=0.5] {$v_2$}
-- ++ (1,-2) coordinate (P4)	node[below left,pos=0.5] {$v_3$}
-- ++ (1.5,0.5) coordinate (P5)	node[below,pos=0.5] {$v_4$}	
-- ++ (0,3) coordinate (P6)  node[pos =0.5,rotate=90] {\tiny$\mathrm{I}$}	;
\foreach \x in {1,2,3,4,5} 
    \fill[] (P\x) circle (2pt);
\end{scope}
\,
\qquad
\begin{scope}[xshift=7.5cm]
  \path[draw,fill=color1]  (4.5,-1) coordinate (P0)
  -- ++ (0,3) coordinate (P1)  node[pos =0.5,rotate=90] {\tiny$\mathrm{II}$}
-- ++ (-1,-1) coordinate (P2) node[above, pos =0.6] {$v_1$}
-- ++ (-1,-0.5) coordinate (P3) node[above,pos=0.5] {$v_2$}
-- ++ (-1,2) coordinate (P4)	node[above right,pos=0.5] {$v_3$}
-- ++ (-1.5,-0.5) coordinate (P5)	node[above,pos=0.5] {$v_4$}	
-- ++ (0,-3) coordinate (P6) node[pos =0.5,rotate=90] {\tiny$\mathrm{II}$};
  \foreach \x in {1,2,3,4,5} 
    \fill[] (P\x) circle (2pt);
\end{scope}
\end{tikzpicture}
\caption{ A hyperelliptic surface with simple poles}
                \label{cap:simplepole}
\end{figure}
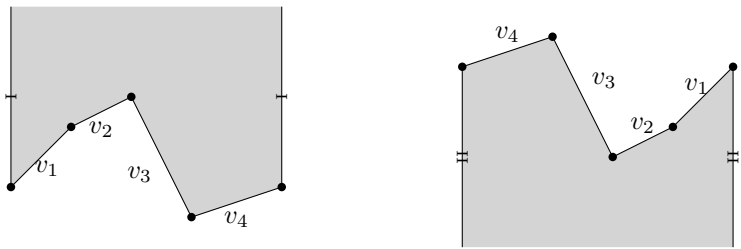
\par
If any of the local equations of $M$ are of the form $\Per(v_i)=a\Per(v_{i+1})$
for some $a\in \bR$, that is, if any two consecutive coordinates~$v_i$ and $v_{i+1}$
are required to be parallel, we obtain a contradiction as follows. As not all
coordinates can be parallel, assume after possibly relabeling them, that we have
three consecutive coordinates, $v_1$ and $v_{2}$ are parallel and not parallel to $v_3$. Further, we can assume that $v_{2}+v_{3}$ is convex (if not, consider the conjugate surface via Lemma~\ref{lem:conjugation}). Then the surface contains cylinders with core curves in directions $v_1+2v_{2}+v_{3}$ and $v_{2}+v_{3}$ which are not parallel and hence cannot both be parallel to the residue core curve. Hence $M$ contains a surface with a cylinder in a different direction than the residue core curve and collapsing the core curve of this cylinder relative to the residue core curve direction provides a contradiction to Proposition~\ref{prop:BDGcrit}.
\par
Now as no two consecutive s are parallel we have two possibilities for any two consecutive coordinates. If $v_i+v_{i+1}$ is convex, then $(X,\omega)$ contains a cylinder in this direction. If $v_i+v_{i+1}$ is concave, then the surface $(X',\omega')$ obtained by complex conjugation of the coordinates also lies in $M$ and contains a cylinder in the direction $\overline{v}_i+\overline{v}_{i+1}$. Since this direction must be parallel to the residue core curve that we have normalized to be real we have $v_i+v_{i+1}=a_i$ for some $a_i\in\bR$ for $i=1,\dots,k-1$, and $v_k+v_1=a_k$ for some $a_k\in\bR$. 
This completes the proof that there are no \Teichmuller curves in the hyperelliptic component of $\omoduli[g,2](2g,-1,-1)$ as in this case $k=2g+1$ is odd and the above necessitates the contradiction that all $v_i\in\bR$ which cannot hold for a surface that generates a \Teichmuller curve.
\par
It remains to consider the hyperelliptic component of $\omoduli[g,2](g,g,-1,-1)$ for $g\geq 1$. In this case $k=2g+2$ and by Lemma~\ref{lem:conjugation} we can assume, without loss of generality, that $v_1+v_2$ is convex. Then as $v_i+v_{i+1}$ is real for all $i$, necessarily $v_{2j+1}+v_{2j+2}$ for $j=0,\dots, g$ are also convex and we have $\sum_{j=0}^{g}a_{2j+1}=1$. Further, after rescaling we may assume $\Im(v_i)=(-1)^i$. 

Let $\psi(t)$ be the action of the parabolic matrix with $t\in\bR$ which fixes the imaginary part of all coordinates and acts on the real part of the coordinate $v_{2j+1}$ for $j=0,\dots,g$ as
$$\psi(t)(\Re(\Per(v_{2j+1}))) \=\Re(\Per(v_{2j+1}))+ t
\qquad (\mod \Per(v_{2j+1}+v_{2j+2})).$$
However, $\sum_{i=1}^k \Per(v_i)=1$ and $M$ is closed, hence $\Per(v_{2j+1}+v_{2j+2})$ is
rational. The same argument after conjugation of the surface holds for $v_{2j}+v_{2j+1}$ for $j=1,\dots,g$ and also for $v_{2g+2}+v_1$. Such conditions precisely describe
the linear manifolds given by the Hurwitz space  $\Hur(d,(g,g,-1,-1))$.
\end{proof}

\subsection{Hyperelliptic components of meromorphic strata with higher
order poles.} \label{sec:higherpole}

Throughout this section we work in a hyperelliptic component
of some signature $\mu$, i.e.\  with one zero or two zeros of equal order
and one pole or two poles of equal order, and suppose moreover that
the pole order is at least two. 
\par
\begin{prop} \label{prop:hypgivescyl}
Every \Teichmuller curve in a meromorphic hyperelliptic stratum contains a flat surface with a cylinder. 
\end{prop}
\par
As a first step toward this claim we prove:
\par
\begin{lemma} \label{le:pullthrough1}
If $(X,\omega)$ generates a \Teichmuller curve~$M$ and has a Boissy presentation where one basic domain
has two saddle connections on its boundary whose homology classes are
independent in~$M$, then there is a flat surface $(X',\omega') \in M$ with
a cylinder.
\end{lemma}
\par
\begin{proof}
Consider the basic domain in the Boissy presentation of $(X,\omega)$ that contains two saddle connections on its boundary whose homology classes are independent in~$M$. 
Hyperellipticity implies that if $v_1,\ldots,v_k$ for $k\geq 2$ are the saddle connections
at the boundary of this half-plane labeled from left to right, then
$v_k,\ldots,v_1$ are the saddle connections at the boundary of  an oppositely oriented half-plane in the Boissy presentation. If the saddle connections $v_i$ and $v_{i+1}$ for some $i=1,\dots,k-1$ form two sides of
a triangle~$\Delta = \Delta(ABC)$ in the core~$\cC(X)$, then this triangle
and its hyperelliptic image gives a cylinder.
\par
If there is no choice of $v_i$ and $v_{i+1}$ that form two sides of a triangle (i.e. the path $v_1,...,v_k$ is concave) then Lemma~\ref{lem:conjugation} gives that the surface obtained by conjugating all periods also lies in $M$. This surface and the assumption that the original domain contained saddle connections with independent directions ensures that some $\bar{v}_i$ and $\bar{v}_{i+1}$ for $i=1,\dots,k-1$ form two sides of a triangle in the core of this conjugated surface. This triangle and its hyperelliptic image give a cylinder.   
\end{proof}
\par
\paragraph{\textbf{Coordinate dancing}} The following argument
to move around the saddle connections on the boundary of a Boissy presentation
also relies on a path within the \Teichmuller curve, but leaving
the $\GL_2(\bR)$-orbit of the initial flat surface.
\par
Pick some saddle connection~$v_k$ on the boundary of some
Boissy basic domain and let~$p$ be the pole corresponding to the (exterior
of this) domain. We write $\Per(v_j) = x_j + i y_j$. Using the action
of $\GL_2(\bR)$
we may arrange that $v_k$ is horizontal, i.e. $\Per(v_k) = x_k$ with $0<x_k<1$.
We now stretch the vertical direction so that $|y_j|>x_j$ for all~$j$ such that
$v_j$ is not horizontal (i.e.\ not parallel to~$v_k$).  We let~$V$ be this
(non-empty) set of indices of ('rather vertical') saddle connections. Consider
now the path that replaces~$\Per(v_k)$ by $\Per_t(v_k) = v_k \exp(-2\pi i t)$
and more generally makes
\ba \label{eq:dancingeq}
\Per_t(v_j):= \int_{v_j} \omega(t) &\= x_j\exp(-2\pi i t) + i y_j \\
&\= x_j \cos(-2\pi i t) + i(x_j \sin(-2\pi i t) +y_j)\,.
\ea
We refer to the path as \emph{coordinate dancing}.  This path stays inside~$M$
since~$M$ is cut out by $\bR$-linear equations. In fact we only have to verify
that those linear equations hold for the real and imaginary parts along the path.
For this we observe that the real part is rescaled by a common factor, while
the imaginary part is the sum of the initial imaginary part and the 
real part rescaled by a common factor, and for those summands the $\bR$-linear
equations hold individually.
\par
The initial phase of the dance, shrinking and rotating till $t=1/4$,
is illustrated in the passage to the second line in Figure~\ref{cap:CoordDancing}
using the saddle connection~$k=2$.
\par
As illustrated in the passage between the third and fourth line of
Figure~\ref{cap:CoordDancing}, once $v_k$ passes the vertical line as we
follow the dancing path, we need to cut and reglue
along~$v_k$ (and simultaneously along all boundary saddle connections
that are parallel to~$v_1$) in order to maintain a Boissy
presentation, resulting in the saddle connection~$v_k'$. (Actually the
fourth row in Figure~\ref{cap:CoordDancing} is not quite a Boissy presentation,
since some of the the saddle connections in~$V$ are 'tilted over'.)
\par
Now two cases may happen. First, none of the saddle connection parallel to~$v_k$
touches any of the rather vertical saddle connections, as shown in the figure.
(In this case the cut and reglue for the 'tilted over' saddle connections in
the set~$V$ will be undone in the next step anyway and so we kept their
position for simplicity.) We thus continue along the dancing path and arrive
at the point where $v_k'$ passes the vertical line (fifth row in
Figure~\ref{cap:CoordDancing}). This requires another cut and reglue  
resulting in the saddle connection~$v_k''$.
\par
\input{pic_nicer_dancing}
\clearpage
To summarize the \emph{first case} (sixth row in Figure~\ref{cap:CoordDancing}):
for $t=1$ all the boundary saddle connections in~$V$ are back to their initial
position, while $v_k$ (and all those parallel to it) have \emph{danced to
the half-plane at angle~$2\pi$} in counterclockwise direction from their initial
position.
\par
In the \emph{second case} some connection parallel to~$v_k$ \emph{touches
in the moment of tilt over} a saddle connection in~$V$. In this case we
have to reglue the tilted over saddle connection to arrive at a Boissy
presentation. The dancing path stops here and the sequel depends on the context.
\par
\begin{lemma} \label{le:coordinatedancing}
Every \Teichmuller curve contains a surface that satisfies the hypothesis of Lemma~\ref{le:pullthrough1}
\end{lemma}
\par
\begin{proof} Let $(X,\omega)$ be a flat surface that generates the \Teichmuller curve $M$. We claim that there is a pole~$p$
so that among the saddle connections on the boundary of the Boissy domains
adjacent to~$p$, there are two which are $\bR$-linearly independent in
homology. Consider the global picture: each boundary saddle connection is
adjacent to one or two poles. Moreover since $(X,\omega)$ generates
a \Teichmuller curve, the boundary saddle connections span a two-dimensional
subspace in cohomology. Connectivity of the whole surface
implies that there is some~$p$ as claimed.
\par
Now we start coordinate dancing with one of the saddle connections adjacent
to~$p$. If the dance stops at the moment of tilt over, we have arrived
at a situation where the hypothesis of Lemma~\ref{le:pullthrough1} is met.
\par
Otherwise, repeating the coordinate dance with the same saddle connection, 
this saddle connection successively visits all half-plane at angle~$m\cdot 2\pi$
for $m \in \bN$ from its initial position until eventually the hypothesis
of Lemma~\ref{le:pullthrough1} is met thanks to the choice of~$p$ and
the presence of another saddle connection.
\end{proof}
\par
The \emph{proof of Proposition~\ref{prop:hypgivescyl}} is an immediate
consequence of Lemma~\ref{le:coordinatedancing} and
Lemma~\ref{le:pullthrough1}.
\par
Thanks to Proposition~\ref{prop:hypgivescyl} we assume from now on
that the surface $(X,\omega)$ that generates the \Teichmuller curve has a cylinder, rotated so that 
its core curve is vertical, in short a \emph{vertical cylinder}.
We work with the generalized Boissy presentation of this surface.
The geometric criterion Proposition~\ref{prop:nonTeichCrit} immediately
implies that if $(X,\omega)$ with a vertical cylinder generates
a \Teichmuller curve, then all other saddle connections on the boundary
of the basic domains in a generalized Boissy presentation are vertical.
We refer such a presentation as a \emph{pure vertical presentation} and
assume this from now on. 
\par
We moreover distinguish the saddle connections on the boundaries of
the building blocks.
\emph{External saddle connections } are the (vertical) saddle connections
on the boundary of the infinite half-planes. \emph{Internal saddle
connections} are the remaining vertical saddle connections. They are
internal to subsurfaces made out of cylinders (with vertical core curves)
only.
\par
Since the hyperelliptic involution acts on the surface, mapping Boissy's
building blocks into building blocks, it also acts on the set of
external and internal saddle connections. These saddle connections are
thus either \emph{invariant} under the action of the involution, or
exchanged in pairs.
\par
\begin{lemma} \label{le:noinvint}
There are no invariant internal saddle connections.
\end{lemma}
\par
\begin{proof}
Any invariant internal coordinate will appear at the two ends of the
same cylinder with vertical core curves. Connecting these saddle connections
by straight lines across the initial cylinder will form the core curves
of another cylinder. Since these crossing lines are not vertical, this
cylinder together with any of the other boundary saddle connections
of the vertical presentation gives the desired contradiction to
Proposition~\ref{prop:nonTeichCrit} (after turning the new direction
into the vertical one).
\end{proof}
\par
We isolate another situation favorable to prove the existence 
of a non-vertical cylinders: A \emph{cleaver} consists of two cylinders
with vertical core curves, a simple cylinder and a cylinder with two or
more saddle connections, with both boundary curves of the simple cylinder
glued to the boundary of the non-simple cylinder as in
Figure~\ref{cap:cleaver}~left. Elementary geometry shows:
\par
\begin{lemma} \label{le:cleaverhasdiag}
  A cleaver contains a cylinder with a non-vertical core curve.
\end{lemma}
\par
\begin{proof} Shearing by a vertical parabolic we may assume that the
open ends of the non-simple cylinders are facing each other, at the expense
of having a skew cleaver. By rescaling we assume that~$S$ in
Figure~\ref{cap:cleaver}~right is a unit square, the cleaver is twisted
by~$0\leq t < 1$, and its hold has length~$a$. Now the
direction of slope $\lfloor t \rfloor /(1+a)$ contains a cylinder.
\end{proof}
\par
\definecolor{color1}{rgb}{0.83, 0.83, 0.83} 
\begin{figure}[htb]
	\centering
\begin{tikzpicture} [scale=.8]

\begin{scope}
\path[draw,fill=color1] ((0,0) coordinate (P1) 
					-- ++ (15:1cm) coordinate (P2) 
					node[] {$\mathrm{I}$}
					-- ++ (15:1cm) coordinate (P3) 
					-- ++ (90:0.5cm) coordinate (P4)	node[right,pos=0.5] {$v_5$}
 					-- ++ (90:0.7cm) coordinate (P5)	node[right,pos=0.5] {$v_4$}	
   					-- ++ (-22.5:2.25cm) coordinate (P6) node[] {$\mathrm{II}$}	
   					-- ++ (-22.5:2.25cm) coordinate (P7)		
					-- ++ (90:0.8cm) coordinate (P8)	node[right,pos=0.5] {$v_3$}	
					-- ++ (157.5:2.25cm) coordinate (P9) node[] {$\mathrm{II}$}
					-- ++ (157.5:2.25cm) coordinate (P10) 
					-- ++ (195:1cm) coordinate (P11) node[] {$\mathrm{I}$}
					-- ++ (195:1cm) coordinate (P12) 
					-- ++ (270:0.5cm) coordinate (P13) node[left,pos=0.5] {$v_1$}
					-- ++ (270:0.7cm) coordinate (P14) node[left,pos=0.5] {$v_2$}  										
					-- ++ (270:0.8cm) coordinate (P15) node[left,pos=0.5] {$v_3$};
\foreach \x in {1,3,4,5,7,8,10,12,13,14,15} 
    \fill[] (P\x) circle (2pt);
\end{scope}
\,
\qquad
\begin{scope}[xshift=7.5cm]
	\path[draw,fill=color1] (0,0) coordinate (P1) 
	-- ++ (0:1cm) coordinate (P2) node[] {$\mathrm{I}$} node[below,yshift=-2pt] {$1$}
	-- ++ (0:1cm) coordinate (P3) 
	-- ++ (25:3.5cm) coordinate (P4)  	
	-- ++ (90:0.375cm) coordinate (P5) node[right] {$v_3$} 
	-- ++ (90:0.375cm) coordinate (P6)
	-- ++ (205:3.5cm) coordinate (P7) 
	-- ++ (90:0.5cm) coordinate (P8)	node[right,pos=0.95] {$v_5$}	
	-- ++ (90:0.7cm) coordinate (P9)	node[right,pos=0.6] {$v_4$}	
	-- ++ (180:1cm) coordinate (P10) node[] {$\mathrm{I}$}
	-- ++ (180:1cm) coordinate (P11) 
	-- ++ (270:0.5cm) coordinate (P12) node[left,pos=0.5] {$v_1$}
	-- ++ (270:0.7cm) coordinate (P13) node[left,pos=0.5] {$v_2$}	
	-- ++ (270:0.75cm) coordinate (P14) node[left,pos=0.5] {$v_3$}; 
	%
	\path[draw,dotted] (P3) -- ++ (0:3.17cm) coordinate (P14) node[below,pos=0.5] {$a$}
							-- ++ (90:1.5cm) coordinate (P15) node[right,pos=0.5] {$t$}
							(P3) -- (P7) -- (P13);

	\foreach \x in {1,3,4,6,7,8,9,11,12,13} 
	\fill[] (P\x) circle (2pt);
	\path[] (P1) -- (P7) node[pos=0.5] {\Large$S$}; 
\end{scope}
\end{tikzpicture}
\caption{ A cleaver and its shear that
  exhibits a horizontal cylinder}
                \label{cap:cleaver}
\end{figure}
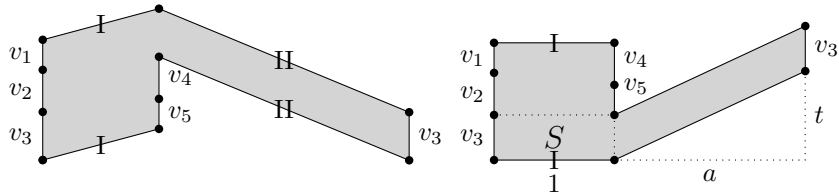
\par
\begin{lemma} \label{le:existschimney}
A surface $(X,\omega)$ that contains a cylinder and generates a \Teichmuller curve, given in 
pure vertical generalized Boissy presentation, consists only of Boissy
basic domains (half-planes) and \emph{simple} cylinders (at least one). The boundary saddle connections are all external, either
invariant or alternatively exchanged on the two ends of a simple cylinder.
\end{lemma}
\par
\begin{proof} By the preceding Lemma~\ref{le:cleaverhasdiag} and our main
criterion Proposition~\ref{prop:nonTeichCrit} we may exclude systems of cylinders
in the complement of the half-planes that contain a cleaver. Since we
excluded invariant internal saddle connections in Lemma~\ref{le:noinvint}
there are four types of cylinders in a generalized Boissy presentation
that we have not yet excluded:
\begin{itemize}
\item[a)] A simple cylinder between two external saddle connections, hence zero
internal saddle connections.
\item[b)] Two external and at least two internal saddle connections.
\item[c)] At least four external and zero or more internal saddle connections.
\item[d)] Zero external and at least four internal saddle connections.
\end{itemize}
(It is in (d) that we can exclude cylinders with two internal saddle
connections, as they will create a cleaver.) 
We denote by the same letter as in the list the number of such cylinders
by~$e$ the number of non-invariant external saddle connections, by
$i$ the number of non-invariant internal saddle connections and by~$f$
the number of invariant ('fixed') external saddle connections.
\par
We count the number of Weierstra\ss\ points. These are the pole (if
unique), the zero (if unique), the~$f$ midpoints of the invariant
(hence external) saddle connections, and two points in the interior of
each cylinder. In fact, in a hyperelliptic component all cylinders are fixed
by the involution, since otherwise we could change the size of one of them,
destroying the existence of the involution while staying in the stratum.
If we write~$Z$ for the number of zeros, $P$ for the number of poles,
and $C$ for the number of cylinders, we get
\bes
2g+2 \= (2-Z) + (2-P) + 2C + f\,.
\ees
Altogether we find the first of the conditions
\bas
a+b+c+d &\= \frac12\bigl(2g-2+Z+P-f\bigr) \\
2a+2b+4c &\,\leq\, e  \\
2b+4d  &\,\leq\, 2\bigl(2g-2-Z+P-f-e\bigr)\,.
\eas
The second condition gives the count of non-invariant external coordinates,
which appear in on opposite sides of the same cylinder (if adjacent to some
cylinder). The third equation gives the count of non-invariant internal
coordinates, each of which is counted twice. The right hand side expresses
that the total number of boundary saddle connections $i+f+e=2g-2+Z+P$ gives
a coordinate system for the stratum $\omoduli[g,n](\mu)$.
\par
Summing twice the second and the third equation implies $b=c=0$. In this
situation the internal coordinates of any element in d) can only be glued
to the same cylinder, since a) does not provide internal coordinates. This
means that the element in d) form their own connected components. Since this
is impossible, we conclude $d=0$ and all the claims of the lemma.
\end{proof}
\par
Next, we constrain how these boundary saddle connections are distributed
on Boissy's basic domains as in Figure~\ref{PullingThrough}, left.
\par
\begin{lemma} \label{le:remainingconfig}
The boundary of each Boissy's basic domains has either
\begin{itemize}
\item[i)] no boundary saddle connections, or
\item[ii)] only invariant external saddle connections, or
\item[iii)] exactly one non-invariant external saddle connection (thus
bounding a simple cylinder on its other side).
\end{itemize}
\end{lemma}
\par
\begin{proof}
Suppose some basic domain has two adjacent non-invariant saddle connections
on its boundary.

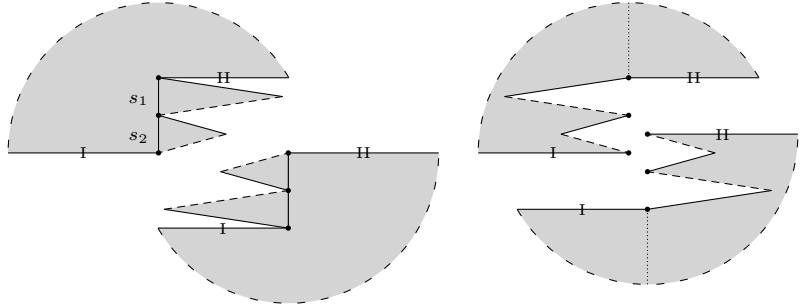
\begin{figure}[htb]
	\centering
\begin{tikzpicture} [scale=.5]
\definecolor{color1}{rgb}{0.83, 0.83, 0.83} 
\begin{scope}
\coordinate (P1) at (0,0);
\coordinate (P2) at (0,1);
\coordinate (P3) at (0,2);
\coordinate (P4) at (3.3,1.5);
\coordinate (P5) at (1.8,0.5);
%
\path[fill=color1] (-4,0) -- (P1) node[pos=0.5] {\tiny$\mathrm{I}$}
-- (P2) -- (P3) -- (30:4) node[pos=0.5] {\tiny$\mathrm{II}$}
 arc (30:180:4);
\path[draw] (-4,0) -- (P1) -- (P2) -- (P3)  -- (30:4);
\path[draw,loosely dashed] (30:4) arc (30:180:4);
%
\path[fill=color1] (P3) -- (P4) -- (P2);
\path[draw] (P2) -- (P3) -- (P4);
\path[draw,densely dashed] (P4) -- (P2);
\path[fill=color1] (P1) -- (P2) -- (P5);
\path[draw] (P1) -- (P2) -- (P5);
\path[draw,densely dashed] (P5) -- (P1);
%
\fill[] (P1) circle (2pt);
\node at (P1) [above left = 0.1mm of P1] {\scriptsize$s_2$};
\fill[] (P2) circle (2pt);
\node at (P2) [above left = 0.1mm of P2] {\scriptsize$s_1$};
\fill[] (P3) circle (2pt);
%
\end{scope}
\begin{scope}[xshift=3.45cm,rotate=-180]
\coordinate (P1) at (0,0);
\coordinate (P2) at (0,1);
\coordinate (P3) at (0,2);
\coordinate (P4) at (3.3,1.5);
\coordinate (P5) at (1.8,0.5);
%
\path[fill=color1] (-4,0) -- (P1) -- (P2) -- (P3) -- (30:4) arc (30:180:4);
\path[draw] (-4,0) -- (P1) node[pos=0.5] {\tiny$\mathrm{II}$}
-- (P2) -- (P3)  -- (30:4) node[pos=0.5] {\tiny$\mathrm{I}$};
\path[draw,loosely dashed] (30:4) arc (30:180:4);
%
\path[fill=color1] (P3) -- (P4) -- (P2);
\path[draw] (P2) -- (P3) -- (P4);
\path[draw,densely dashed] (P4) -- (P2);
\path[fill=color1] (P1) -- (P2) -- (P5);
\path[draw] (P1) -- (P2) -- (P5);
\path[draw,densely dashed] (P5) -- (P1);
%
\fill[] (P1) circle (2pt);
\node at (P1) [below right = 0.5mm of P1] {\scriptsize$ $};
\fill[] (P2) circle (2pt);
\node at (P2) [below right = 0.5mm of P2] {\scriptsize$ $};
\fill[] (P3) circle (2pt);
%
%
\end{scope}

\begin{scope}[xshift=12.5cm]
	\begin{scope}
	
	\coordinate (P1) at (0,0);
	\coordinate (P2) at (0,1);
	\coordinate (P3) at (0,2);
	\coordinate (P4) at (-3.3,1.5);
	\coordinate (P5) at (-1.8,0.5);
        \coordinate (P6) at (0,4);
	%
	\path[fill=color1] (-4,0) -- (P1) -- (P5) -- (P2) -- (P4) -- (P3) -- (30:4) arc (30:180:4);
	
	\path[draw] (-4,0) -- (P1) node[pos=0.5] {\tiny$\mathrm{I}$}
	(P5) -- (P2) (P4) -- (P3) -- (30:4) node[pos=0.5] {\tiny$\mathrm{II}$};
	\path[draw,densely dashed] (P4) -- (P2) (P5) -- (P1);
	\path[draw,loosely dashed] (30:4) arc (30:180:4);
        \path[draw,densely dotted] (P3) -- (P6);

	%
	\fill[] (P1) circle (2pt);
	\fill[] (P2) circle (2pt);
	\fill[] (P3) circle (2pt);
	%
	\end{scope}
	\begin{scope}[xshift=.5cm,yshift=.5cm,rotate=-180]	
	\coordinate (P1) at (0,0);
	\coordinate (P2) at (0,1);
	\coordinate (P3) at (0,2);
	\coordinate (P4) at (-3.3,1.5);
	\coordinate (P5) at (-1.8,0.5);
        \coordinate (P7) at (0,4);
	%
	\path[fill=color1] (-4,0) -- (P1) -- (P5) -- (P2) -- (P4) -- (P3) -- (30:4) arc (30:180:4);
	
	\path[draw] (-4,0) -- (P1) node[pos=0.5] {\tiny$\mathrm{II}$}
	(P5) -- (P2) (P4) -- (P3)  -- (30:4) node[pos=0.5] {\tiny$\mathrm{I}$};
	\path[draw,densely dashed] (P4) -- (P2) (P5) -- (P1);
	\path[draw,loosely dashed] (30:4) arc (30:180:4);
        \path[draw,densely dotted] (P3) -- (P7);
	%
	\fill[] (P1) circle (2pt);
	\fill[] (P2) circle (2pt);
	\fill[] (P3) circle (2pt);
	%
%
\end{scope}		
\end{scope}
\end{tikzpicture}
\caption{Pulling through cylinders: The surface $\Phi(0)$ on the left and
$\Phi(1)$ on the right}
 \label{PullingThrough}
\end{figure}	
%
Shearing by a vertical parabolic element we may assume that those simple
cylinder with vertical core curves bounding the non-invariant boundary
saddle connections do not have horizontal saddle connections. By cut and
re-glue we may present the cylinders to be twisted by less that the length of
the waist curve, i.e. such that the diagonal decomposes the cylinder into two
acute triangles, see the left of Figure~\ref{PullingThrough}. This figure also
shows the new coordinates that we are using. The coordinate system no longer
contains the waist curves of the cylinders~$s_i$, but the sides of the acute
triangles $v_1,\ldots,v_4$  and (if they exist) the remaining coordinates
$v_5,\ldots,v_n$ on the boundary of the half planes.
\par
We consider the 'pulling through' path $\phi(t) = (X_t,\omega_t)$
as in~\eqref{eq:pullthrough}. For $t \in [0,1]$. This path stays within the
complex two-dimensional manifold containing the $\GL_2^+(\bR)$-orbit of the
initial surface. At its endpoint $t=1$ we construct the desired cylinder as
drawn on the right of Figure~\ref{FindingCylinders}.
Proposition~\ref{prop:nonTeichCrit} applies, giving a contradiction.
\par

\begin{figure}[htb]
	\centering
\begin{tikzpicture} [scale=.5]
\definecolor{color1}{rgb}{0.83, 0.83, 0.83} 
\begin{scope}
	\begin{scope}
	
	\coordinate (P0) at (-4,0);
	\coordinate (P1) at (0,0);
	\coordinate (P2) at (0,.6);
	\coordinate (P3) at (0,1.2);
	\coordinate (P4) at (0,1.8);
	\coordinate (P5) at (0,2.4);		
	\coordinate (P6) at (-2.7,2.5);
	\coordinate (P7) at (0,3.1);
	\coordinate (P8) at (-4,3.1);
	%
	\path[fill=color1] (P0) -- (P1) -- (P2) -- (P3) -- (P4) -- (P5) -- (P6) -- (P7) -- (50:4) arc (30:179:3.5) -- (P0);
	\path[draw] (P0) -- (P1)   node[pos=0.5] {\tiny$\mathrm{I}$}
	node[left,xshift=3.5,yshift=4] {\tiny $v_1$} -- (P2) 
	                         node[left,xshift=3.5,yshift=4] {\tiny $v_2$} -- (P3)
	                         node[left,xshift=3.5,yshift=4] {\tiny $v_3$} -- (P4)
	                         node[left,xshift=3.5,yshift=4] {\tiny $v_4$} -- (P5) 
	                         node[left,yshift=3.5] {\tiny $v_5$} -- (P6) -- (P7)
	                         node[left,xshift=-5,yshift=1] {\tiny $v_6$}  -- (50:4)
	                         node[pos=0.5] {\tiny$\mathrm{II}$}
	                         ;

	\path[draw,densely dotted] (P4) -- (P6);
	\path[draw,loosely dashed] (50:4) arc (30:180:3.5) -- (P0);
	\fill[] (P0) circle (2pt);	
	\fill[] (P1) circle (2pt);
	\fill[] (P2) circle (2pt);
	\fill[] (P3) circle (2pt);
	\fill[] (P4) circle (2pt);
	\fill[] (P5) circle (2pt);
	\fill[] (P6) circle (2pt);
	\fill[] (P7) circle (2pt);		
	%
	
	\end{scope}

	\begin{scope}[xshift=1cm,yshift=1.7cm,rotate=-180]		
	\coordinate (P0) at (-4,0);
	\coordinate (P1) at (0,0);
	\coordinate (P2) at (0,.6);
	\coordinate (P3) at (0,1.2);
	\coordinate (P4) at (0,1.8);
	\coordinate (P5) at (0,2.4);		
	\coordinate (P6) at (-2.7,2.5);
	\coordinate (P7) at (0,3.1);
	%
	\path[fill=color1] (P0) -- (P1) -- (P2) -- (P3) -- (P4) -- (P5) -- (P6) -- (P7) -- (50:4) arc (30:179:3.5) -- (P0);
	\path[draw] (P0) -- (P1) 
	node[pos=0.5] {\tiny$\mathrm{II}$}
	node[right,xshift=-3,yshift=-4.2] {\tiny $v_1$} -- (P2) 
	node[right,xshift=-3,yshift=-4.2] {\tiny $v_2$} -- (P3)
	node[right,xshift=-3,yshift=-4.2] {\tiny $v_3$} -- (P4)
	node[right,xshift=-3,yshift=-4.8] {\tiny $v_4$} -- (P5) 
	node[right,xshift=2,yshift=-4] {\tiny $v_5$} -- (P6) -- (P7)
	node[right,xshift=10,yshift=-1] {\tiny $v_6$}  -- (50:4)
	node[pos=0.5] {\tiny$\mathrm{I}$};
	
	\path[draw,densely dotted] (P4) -- (P6);
	\path[draw,loosely dashed] (50:4) arc (30:180:3.5) -- (P0);
	\fill[] (P0) circle (2pt);	
	\fill[] (P1) circle (2pt);
	\fill[] (P2) circle (2pt);
	\fill[] (P3) circle (2pt);
	\fill[] (P4) circle (2pt);
	\fill[] (P5) circle (2pt);
	\fill[] (P6) circle (2pt);
	\fill[] (P7) circle (2pt);		
	%

\end{scope}

\end{scope}
\begin{scope}[xshift=10.5cm, yshift=-1.2cm]
	\begin{scope}[yshift=2cm]
		
		\coordinate (P0) at (-4,0);
		\coordinate (P1) at (0,0);
		\coordinate (P2) at (-1.5,.3);
		\coordinate (P3) at (0,1.3);
		\coordinate (P4) at (-2.7,1.3);
		\coordinate (P5) at (0,2.3);		

		%
		
		\path[fill=color1] (P0) -- (P1) -- (P2) -- (P3) -- (P4) -- (P5) -- (35:4) arc (35:180:4);
		%
		
		\path[draw] (P0) -- (P1) 
		node[pos=0.5] {\tiny$\mathrm{I}$}
		node[left,xshift=3.5,yshift=4] {\tiny $v_4$} -- (P2) 
		-- (P3) node[below,xshift=-2.5,yshift=-2] {\tiny $v_3$}
		node[above,xshift=-5,yshift=-2.5] {\tiny $v_2$} -- (P4)
		-- (P5) node[left,xshift=-8,yshift=-1] {\tiny $v_1$} -- (35:4)
		node[pos=0.5] {\tiny$\mathrm{II}$};
		
		\path[draw,densely dotted] (P2) -- (P4);
		\path[pattern=north east lines] (P2) -- (P3) -- (P4);
		\path[draw,loosely dashed] (35:4) arc (35:180:4); 
		\fill[] (P0) circle (2pt);	
		\fill[] (P1) circle (2pt);
		\fill[] (P2) circle (2pt);
		\fill[] (P3) circle (2pt);
		\fill[] (P4) circle (2pt);
		\fill[] (P5) circle (2pt);
		\fill[] (P6) circle (2pt);
		\fill[] (P7) circle (2pt);		
		%
	\end{scope}
	\begin{scope}[xshift=1cm,yshift=2.1cm,rotate=-180]
	
	\coordinate (P0) at (-4,0);
	\coordinate (P1) at (0,0);
	\coordinate (P2) at (-1.5,.3);
	\coordinate (P3) at (0,1.3);
	\coordinate (P4) at (-2.7,1.3);
	\coordinate (P5) at (0,2.3);		
	
	%
	\path[fill=color1] (P0) -- (P1) -- (P2) -- (P3) -- (P4) -- (P5) -- (35:4) arc (35:180:4);
	\path[draw] (P0) -- (P1) 
	node[pos=0.5] {\tiny$\mathrm{II}$}
	node[below,xshift=3.5,yshift=1] {\tiny $v_4$} -- (P2) 
	-- (P3) node[below,xshift=8,yshift=1] {\tiny $v_2$}
	node[above,xshift=3,yshift=1] {\tiny $v_3$} -- (P4)
	-- (P5) node[right,xshift=9
	,yshift=1] {\tiny $v_1$} -- (35:4)
	node[pos=0.5] {\tiny$\mathrm{I}$};
	
	\path[draw,densely dotted] (P2) -- (P4);
	\path[pattern=north east lines] (P2) -- (P3) -- (P4);
	\path[draw,loosely dashed] (35:4) arc (35:180:4);
	\fill[] (P0) circle (2pt);	
	\fill[] (P1) circle (2pt);
	\fill[] (P2) circle (2pt);
	\fill[] (P3) circle (2pt);
	\fill[] (P4) circle (2pt);
	\fill[] (P5) circle (2pt);
	\fill[] (P6) circle (2pt);
	\fill[] (P7) circle (2pt);		
	%
\end{scope}
\end{scope}
\end{tikzpicture}
 \caption{Finding cylinders}
 \label{FindingCylinders}
\end{figure}
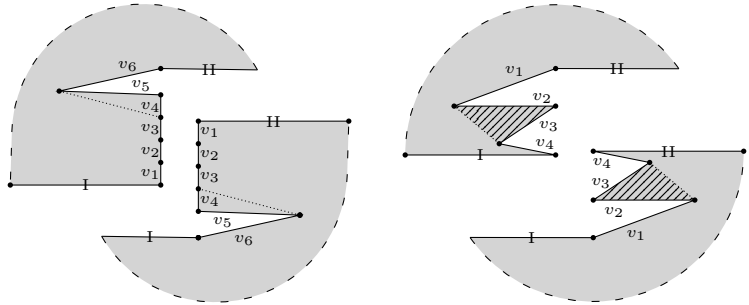	
The other case to rule out is (at least) one simple cylinder (with its
adjacent)  non-invariant external saddle connections next to some 
invariant external saddle connection. By 'pulling through' as in the previous
case we arrive at the configuration in Figure~\ref{FindingCylinders}
on the left. Then there is a cylinder having this vertical coordinate as one
of its diagonals, and together with the saddle connection bounding the
cut-out acute triangle that does not lie inside the cylinder we have the
configuration that provides a contradiction by applying
Proposition~\ref{prop:nonTeichCrit}.
\end{proof}
\par
\medskip
\paragraph{\textbf{Rotate again for the next dance}} So far, for
Lemma~\ref{le:remainingconfig} we have been using the generalized Boissy presentation,
i.e.\ admitting cylinders with vertical core curves. In order to get further
constraints on the periods of the saddle connections we have to use again paths
in the linear manifold that leave the $\GL_2(\bR)$-orbit, i.e.\ dancing paths
or a variant of it. For this purpose consider a surface with in
Lemma~\ref{le:remainingconfig} and pull through the cylinders, like
in Figure~\ref{PullingThrough} right (but with only one cylinder or one invariant
external saddle connection per basic domain). Cutting the basic domains
open along vertical lines at the end of the broken line of saddle connections
(the dotted lines in Figure~\ref{PullingThrough} right), regluing along the
infinite horizontal separatricies  and rotating the picture by~$\pi$ gives
a flat surfaces in (standard) Boissy presentation that we use next.
\par
\begin{lemma}
Suppose the surface~$(X,\omega)$ in the presentation of Lemma~\ref{le:remainingconfig}
has an invariant external saddle connections (as in (ii) of that lemma).
Then the \Teichmuller curve generated by $(X,\omega)$ is one of the obvious
cases from Proposition~\ref{prop:obviousTeich}  (iii a) or (iii b) covering surfaces
with residues conditions.
\end{lemma}
\par
\begin{proof}
We may assume by Proposition~\ref{prop:hypgivescyl} that there is at least one cylinder.
We pull this cylinder through and rotate for the next dance, as described above.
We select one of the invariant external saddle
connections~$v_1$ to apply the coordinate dancing. Since $(X,\omega)$ generates
a \Teichmuller curve, all the invariant external saddle connections are parallel and hence
dance simultaneously and if there are several of them adjacent to one Boissy domain,
they dance (i.e.\ move between the Boissy domains) together. This dance will terminate
in one of two ways: 
\par
The first possibility is that a set of invariant coordinates appears on a basic
domain with two pulled-through coordinates at the end of a certain number of
full dancing paths (i.e., full rotations $2k\pi$ for some $k\in \bN$). This is a contradiction, since we have three or more non-parallel
coordinates on a basic domain, as in Figure~\ref{FindingCylinders} left
(erasing $v_2,v_3,v_4$ to get the simplest case).
\par
The second possibility is the dance ends the moment the cylinder coordinates
are tilted over (i.e., rotations by $(2k+1)\pi $ for some $k\in \bN$). One pulled-through coordinate and the adjacent invariant coordinate
now form a cylinder. Hence if the set of invariant coordinates adjacent to
any Boissy domain contains more than one coordinate we obtain a contradiction. 
\par
Let $v_2$ and $v_3$ be the saddle connections that have been pulled through.
On the original surface $v_2 + v_3$ is the waist curve of the cylinder, so there
must be a relation $\Per(v_2 + v_3 - sv_1) = 0$ for some $s \in \bR$, since otherwise
Proposition~\ref{prop:nonTeichCrit} provides a contradiction. At the end of the
dance $-v_1 + v_2$ is the waist curve of a cylinder (possibly after changing the
role of $v_2$ and $v_3$). This implies that  $v_3 \in \bR \cdot (-v_1 + v_2)$
and thus (possibly changing the orientation of $v_1$) we find $s=1$.
\par
Now we consider the surface globally.  If there were any more invariant coordinates
they cannot appear at the end of the dance in their own basic domains as this would
form a contradiction (by applying Proposition~\ref{prop:nonTeichCrit} to the
newly formed cylinder). Similarly, if there were any other pulled-through coordinates
they will now appear in "tilted-over" position. Using them as saddle connections
and the newly formed cylinder, Proposition~\ref{prop:nonTeichCrit} will 
provide a contradiction unless one of the pulled-through coordinates is now parallel
to the waist curve of the new cylinder (that is, to~$v_3$) and the other forming
a new cylinder with an invariant coordinate with necessarily parallel waist curve
to $v_1$. 
\par
Labeling these coordinates $u_1,u_2,u_3$ consistently with the $v_i$ we find
$u_1$ and $v_i$ are parallel and in fact, $u_i= kv_i$ for some $k>0$.
\par
Now we label the (even number) Boissy domains adjacent to a poles in counterclockwise
order, starting with a pulled-through cylinder on domain number one. Since the
'first possibility' above provided a contradiction, the first Boissy domain with an
invariant saddle connection has an even number. Rotating in the other direction
we find the the next Boissy domain with pulled-through cylinders sits on a Boissy
domain with an even number and inductively we see that this parity constraint
holds for all cylinder and all invariant saddle connections. Comparing the
triples of saddle connections for a pulled-through cylinder and an invariant saddle
connection for clockwise and counterclockwise dancing implies that all pairwise
comparison factors (denoted by~$k$ above) are equal to one, i.e.\ any two
invariant saddle connections and any two cylinders have the same geometry.
\par
The parity constraint for the Boissy domains containing a (don't pull it through!)
cylinder implies that the cylinders cannot connect Boissy domains adjacent to
the same poles. Consequently, there are precisely two poles since we are in
a hyperelliptic stratum.
\par
The parity constraint (together with the length agreements that follow from $k=1$)
also implies that the map stacking all the cylinders on
top of each other (or equivalently all the invariant external saddle connections
on top of each other) extends to a well-defined covering map. Depending on the
degree of the map we are in case (iii a) or (iii b).
\end{proof}
\par
\begin{lemma}
Suppose the surface~$(X,\omega)$ in the presentation of Lemma~\ref{le:remainingconfig}
has no invariant external saddle connections (as in (ii) of that lemma).
Then the \Teichmuller curve~$M$ generated by $(X,\omega)$ is one of the obvious
cases from Proposition~\ref{prop:obviousTeich} (i a), (ii a), or (ii b).
\end{lemma}
\par
\begin{proof}
We may assume by Proposition~\ref{prop:hypgivescyl} that there is at least
one cylinder. We pull this cylinder through and rotate for the next dance,
as described previously. The situation is illustrated in the first row of
Figure~\ref{cap:SteepestDancing}. We now select a steepest saddle connection~$v_i$
for dancing, i.e., one for which $\Im(\Per(v_i))/\Re(\Per(v_i))$ is maximal.
The dancing path is easier to visualize if we normalize this coordinate to
be vertical at the start of the dance (see the coordinate~$v_i$ in the
second row of Figure~\ref{cap:SteepestDancing}), as opposed to horizontal
in the original coordinate dancing. (In fact with this choice any saddle connection
not parallel to the selected one will remain adjacent to its initial basic
domain throughout the dance.) We stretch sufficiently in the horizontal direction
(such that $\Re(\Per(v_j)) > |\Per(v_k)|$ for all~$k$ with $v_k$ parallel to~$v_i$
and for all~$j$ in the complementary set~$H$ of ('rather horizontal') saddle
connections).
\par
Now we rotate $\Per(v_i)$ while modifying slightly the periods of the saddle
connections so as to stay in~$M$, just as in~\eqref{eq:dancingeq} with the
role of real and imaginary part swapped. The dancing procedure is illustrated
in the remaining rows of  Figure~\ref{cap:SteepestDancing}). It ends once the
selected saddle connection  appears on the same basic domain as some other saddle
connection. If a dancing coordinate arrives on a domain with two coordinates
we have a contradiction by the three coordinates on a domain argument.
\par
Hence the dance must terminate with a dancing coordinate on a domain with
just one other coordinate. However, this means that the other coordinate that
originated on this domain must have danced and was hence parallel to our steepest saddle connection. If that saddle connection now lies in
a domain by itself we obtain a contradiction as we obtain a cylinder (after
possibly pulling-through) with waist curve not parallel to this saddle connection.
Hence this saddle connection coordinate must have danced to a domain with
exactly one other saddle connection. However, this means a saddle connection
that originally lay on this domain must have danced. Chasing this argument around
the surface we see that every cylinder contributed a pulled-through saddle
connection that danced and hence was parallel to the selected saddle
connection.
\par
Label the saddle connections  $v_1,\ldots,v_d$ where $d=2g+|\mu|-2$ such that
$v_{2j-1}$ and $v_{2j}$ for $j=1,\ldots,d/2$ are the pulled-through coordinates from
each cylinder from left to right on the upper basic domains, and such that the
dance results in $v_{2j-1}$ appearing on the domain with $v_{2j+2}$ (considering the
indices modulo $d$). Our argument above gives $v_{2j-1}=k_jv_1$ for $k_j\in \bR$ and
$j=2,\ldots,d/2$. The same argument dancing instead saddle connections in the
direction such that $\Im(\Per(v_i))/\Re(\Per(v_i))$ is minimal gives $v_{2j}=\ell_j
v_2$ for $\ell_j\in \bR$ and $j=2,\ldots,d/2$.
\par
\input{pic_last_dance}
\clearpage
In $(X,\omega)$ the waist curves were necessarily parallel and hence imposed
the condition that $v_{2j-1}+v_{2j}=m_j(v_1+v_2)$ for $m_j\in \bR$. Similarly, after
dancing we create cylinders by pulling through which implies
$v_{2j-1}+v_{2j+2}=n_j(v_1+v_4)$ for $n_j\in \bR$ and $j=1,...,d/2$. 
\par
The only solution to these equations is $k_j=l_j=m_j=n_j=1$ for all~$j$.
Thus all the flat surfaces parametrized by~$M$ admit a cover onto a flat
surface with just two saddle connections on the boundary of the Boissy domains.
Hence the \Teichmuller curve generated by $(X,\omega)$ is one of the obvious
cases as listed in the proposition, the cases being distinguished by the
number of zeros (and poles) in the range of the covering map and the degree
of the cover.
\end{proof}
\par
The \emph{proof of Theorem~\ref{thm:main}} is complete as a combination
of Proposition~\ref{prop:HYPtwosimplepoles} in Section~\ref{sec:simplepole}
and the series of lemmas in this Section~\ref{sec:higherpole}.

%
%
%



\makeatletter
\def\part{\@startsection{part}{0}%
\z@{\linespacing\@plus\linespacing}{.5\linespacing}%
  {\normalfont\Large\bfseries\raggedright}}
\makeatother
\part*{Appendix: An $\bR$-linear non-algebraic manifold}
\noindent by \textsc{Benjamin Bakker}
and \textsc{Scott Mullane}


\
\par
\bigskip
An $\bR$-linear manifold is a submanifold of a stratum of differentials defined locally by homogenous real linear equations in period coordinates. They hold great importance stemming from a diverse range of connections including billiards in polygons, Jacobians with real multiplication, and dynamical rigidity. 
In this appendix, we present a simple example of $\bR$-linear manifold in a meromorphic stratum that is not algebraic, hence showing the algebraicity of these loci in holomorphic strata~\cite{Filip} does not extend to the meromorphic case.
\smallskip

For $\mu=(m_1,\dots,m_n)$ an integer partition of $2g-2$, the stratum of differentials of type $\mu$ is the moduli space of \emph{flat surfaces} or pairs of pointed smooth curves and meromorphic differential of type $\mu$, set theoretically,
$$\Omega\mathcal{M}_{g,n}(\mu):=\{(X,\omega,p_1,\dots,p_n)\hspace{0.1cm}|\hspace{0.1cm}(X,p_1,\dots,p_n)\in\mathcal{M}_{g,n}, \hspace{0.1cm} (\omega)_0-(\omega)_\infty=\sum m_ip_i\}.$$
While the strata inherit an algebraic structure as a stratification of the Hodge bundle, integrating the differential $\omega$ yields a presentation of~$X$, punctured at the poles of $\omega$, as polygons in the complex plane with parallel side identifications. Hence we obtain a complex analytic orbifold structure locally at a point from a choice of basis for the relative homology $H_1(X\backslash P,Z,\bZ)$, where $Z$ and $P$ are the zeros and poles of $\omega$ respectively. This basis extends locally via the flat connection and local orbifold coordinates known as \emph{period coordinates} are obtained by integrating the differential by this basis. 
\begin{figure}[htbp]
\begin{center}
\begin{overpic}[width=1\textwidth]{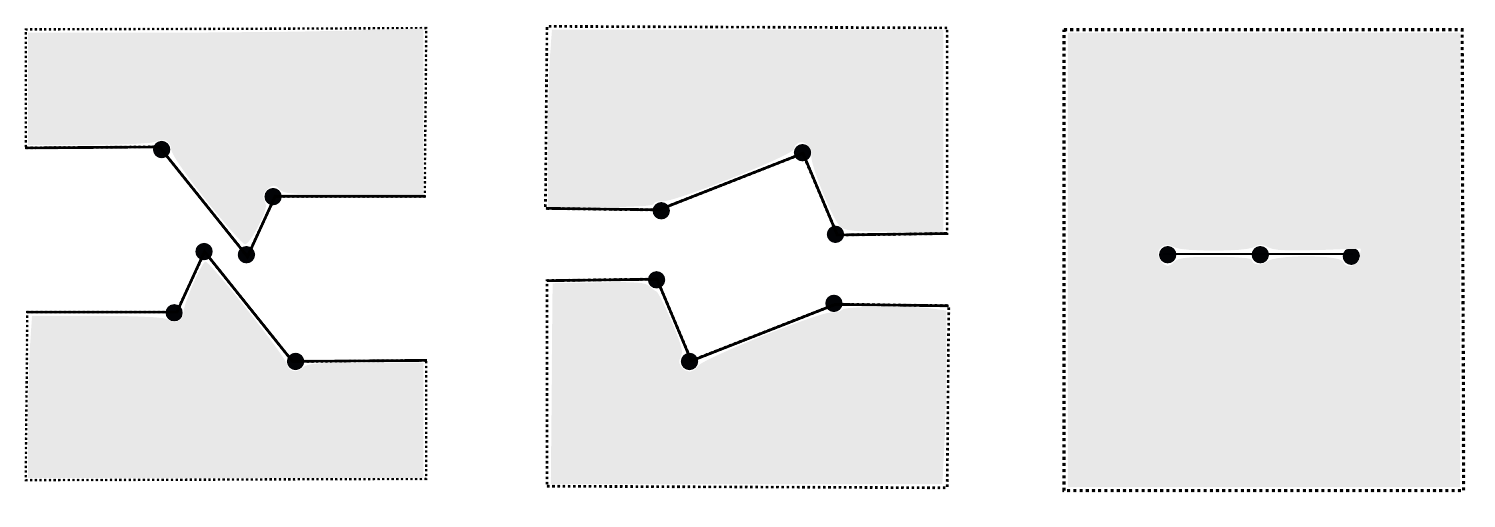}
\put(10.8,21.7){$a$}
\put(10.2,16.3){$b$}
\put(17.5,15){$a$}
\put(18.5,19){$b$}

\put(47.2,23.7){$a$}
\put(43.2,12.5){$b$}
\put(51.5,11.5){$a$}
\put(55.5,22.2){$b$}

\put(80.5,19.5){$a$}
\put(86.5,19.5){$b$}
\put(80.5,16){$b$}
\put(86.5,16){$a$}
\end{overpic}
 \caption{Three flat surfaces in $\Omega\mathcal{M}_{1,2}(2,-2)$}
  \label{Fig1}
\end{center}
\end{figure}

For example, Figure~\ref{Fig1} contains polygon presentations for three different flat surfaces in the stratum $\Omega\mathcal{M}_{1,2}(2,-2)$. In this case, both $Z$ and $P$ are one-point sets and a flat surface in the stratum can be expressed by two broken half-planes with parallel side identifications as follows. The pairs of line segments $a$ and $b$ are identified as labeled and the two infinite half rays extending to the left and the two infinite half rays extending to the right are identified respectively. This gives one vertex with cone angle $6\pi$, the unique double zero of the differential,  and the surface is punctured at infinity, the unique double pole of the differential. 
Varying $a$ and $b$ in $\bC$ with the condition $\Re(a),\Re(b)>0$, provides period coordinates for a local chart in the stratum $\Omega\mathcal{M}_{1,2}(2,-2)$. The flat surface on the right  in Figure~\ref{Fig1} is obtained by setting $a=b=1$ in this chart and identifying the infinite half rays. The polygon presentation for this flat surface is then the infinite plane with two \emph{slits} with opposite sides identified as labeled. 

Now consider the stratum $\Omega\mathcal{M}_{1,4}(2,-2,0,0)$ obtained by further allowing two phantom zeros and let $\pi$ be the forgetful map to $\Omega\mathcal{M}_{1,2}(2,-2)$ that forgets these two points.
We obtain period coordinates for the fibre of $\pi$ over the flat surface on the right in Figure~\ref{Fig1} as $(\bC\backslash\{0,1,2\})^2\backslash\Delta$ where $\Delta$ denotes the diagonal and we identify $a$ and $b$ with the open real intervals $(0,1)$ and $(1,2)$ respectively.  Setting $u-v=1$ we obtain a local $\bR$-linear condition and we are left to consider the closure. The only monodromy is obtained by passing the phantom zeros through the passage of slits, that is, the coordinates $(u,v)$ change as a point passes through the passage $(0,1)\cup(1,2)$. For example, Figure~\ref{Fig2} shows how the the loci $u-v=k$ and $u-v=k+1$ are connected for any $k\in \bZ\backslash\{0\}$ and $k\ne -1$. Further, the loci $u-v=-1$ and $u-v=1$ are connected by passing both phantom zeros through the slits simultaneously. Hence we obtain an irreducible $\bR$-linear manifold $\mathcal{T}$ in the stratum $\Omega\mathcal{M}_{1,2}(2,-2,0,0)$ as the $\bC^*$-orbit of the loci cut out by $u-v\in \bZ\backslash\{0\}$.
\begin{figure}[htbp]
\begin{center}
\begin{overpic}[width=1\textwidth]{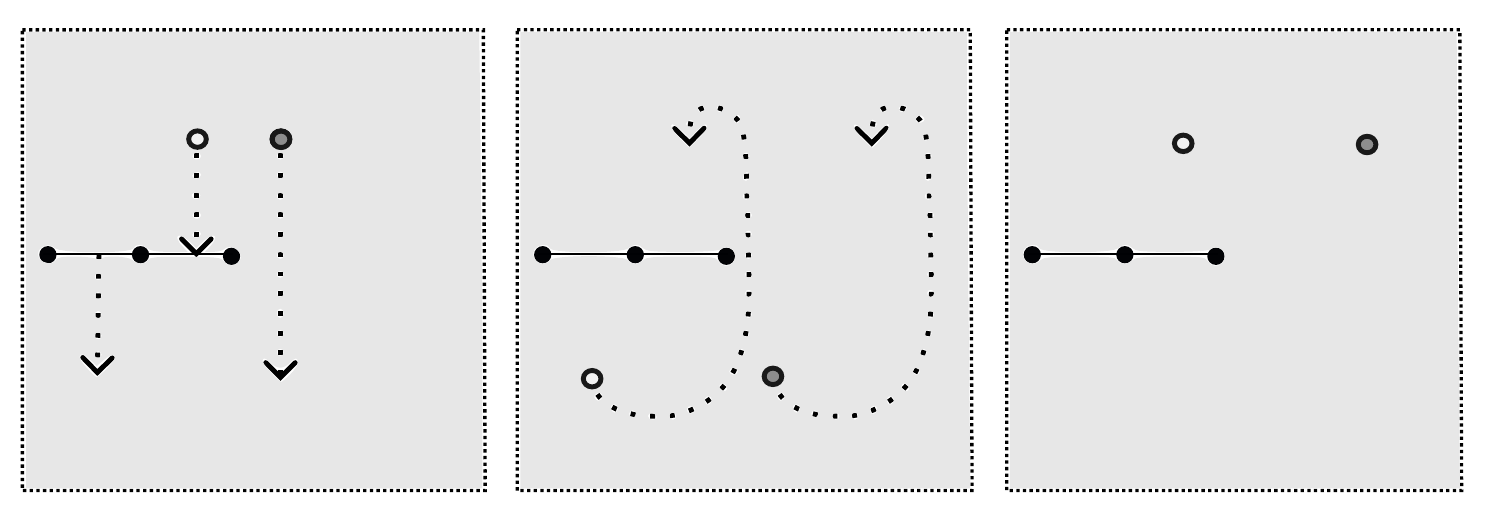}

\end{overpic}
 \caption{Inside the linear manifold $\mathcal{T}$}
  \label{Fig2}
\end{center}
\end{figure}

However, the forgetful map factors as $\pi=\pi_1\circ\pi_2$ through the stratum with just one phantom zero, that is, forgetting first $p_4$ and then $p_3$:
$$\Omega\mathcal{M}_{1,2}(2,-2,0,0)\xrightarrow{\pi_1} \Omega\mathcal{M}_{1,2}(2,-2,0)\xrightarrow{\pi_2} \Omega\mathcal{M}_{1,2}(2,-2)$$
Fixing $a=b=1$ and $u=c\in\bC\backslash\{0,1,2\}$ and allowing $v$ to vary in $\bC\backslash\{0,1,2,c\}$ we obtain a fibre of $\pi_2$ which is hence algebraic. However, this fibre intersects $\mathcal{T}$ in infinitely many points given by $u=c$, $a=b=1$, $v=c+k\ne 0,1,2$ for $k\in\bZ\backslash\{0\}$ contradicting the algebraicity of $\mathcal{T}$.



\printbibliography

\end{document}